\documentclass[12pt]{amsart}

\usepackage{url,verbatim,hyperref}

\usepackage{amsmath,amstext,amssymb,amsopn,amsthm}
\usepackage{url,verbatim}
\usepackage{mathtools}
\usepackage{enumerate}

\usepackage{color,graphicx}

\usepackage[margin=30mm]{geometry}
\usepackage{eucal,mathrsfs,dsfont}

\allowdisplaybreaks

\newtheorem{theorem}{Theorem}[section]

\newtheorem{lemma}[theorem]{Lemma}
\newtheorem{proposition}[theorem]{Proposition}

\theoremstyle{definition}

\newtheorem{definition}[theorem]{Definition}

\newtheorem{remark}[theorem]{Remark}

\theoremstyle{remark}

\numberwithin{equation}{section}

\newcommand{\eps}{\varepsilon}

\newcommand{\calF}{\mathcal{F}}

\newcommand{\calD}{\mathcal{D}}
\newcommand{\calP}{\mathcal{P}}

\newcommand{\calS}{\mathcal{S}}

\newcommand{\calC}{\mathcal{C}}

\renewcommand{\P}{\operatorname{\mathds{P}}} 
\newcommand{\R}{\mathds{R}}

\newcommand{\ol}{\overline}
\newcommand{\wh}{\widehat}

\newcommand{\wt}{\widetilde}

\newcommand{\norm}[1]{\Vert #1 \Vert}

\title[Skorokhod problem]{Uniqueness for the Skorokhod problem
in an orthant: critical cases}
\author{Richard F. Bass and Krzysztof Burdzy}

\address{RFB: Department of Mathematics, University of Connecticut,
Storrs, CT 06269-1009}
\email{r.bass@uconn.edu}

\address{KB: Department of Mathematics, Box 354350, University of Washington, Seattle, WA 98195}
\email{burdzy@uw.edu}

\thanks{Research supported in part by Simons Foundation Grant 928958. }

\begin{document}

\begin{abstract}
Consider  the Skorokhod problem in the closed non-negative orthant:
find a solution $(g(t),m(t))$ to
\[ g(t)= f(t)+ Rm(t),\]
where $f$ is a given continuous vector-valued function with $f(0)$ in the orthant, 
$R$ is a given $d\times d$ matrix with 1's along the diagonal, 
$g$ takes values in the orthant, and $m$ is a vector-valued function
that starts at 0, each component of $m$ is non-decreasing and continuous, and 
for each $i$ the
$i^{th}$ coordinate of $m$ increases only when the $i^{th}$ coordinate of 
$g$ is 0. The stochastic version of the Skorokhod problem replaces $f$ by the 
paths of Brownian motion.
It is known that there exists a unique solution to the Skorokhod problem
if the
spectral radius of $|Q|$ is less than 1, where $Q=I-R$ and $|Q|$ is
the matrix whose entries are the absolute values of the corresponding 
entries of $Q$. The first result of  this paper shows pathwise uniqueness 
for the stochastic version of the Skorokhod problem holds
if the spectral radius of $|Q|$ is equal to 1.
The second result of this paper settles the remaining open cases for 
uniqueness for the deterministic version when the dimension $d$ is  two.
\end{abstract}

\maketitle

\section{Introduction}\label{sect-intro}

We consider both the deterministic and stochastic versions of 
the Skorokhod problem in the $d$-dimensional 
orthant $D=\{(x_1, \ldots, x_d): x_i\ge 0, i=1, \ldots, d\}$ with oblique reflection on the boundary
of $D$ that is constant on each face of $D$.

First we define the deterministic Skorokhod problem.
Let  $R$ be a fixed $d\times d$ matrix.

\begin{definition}\label{intro-SP} {A driving function $f$  is a continuous function
from $[0,\infty)$ to $\R^d$ with $f(0)\in D$. The deterministic
Skorokhod problem is to find $g(t)$ and $m(t)$ so that\\
(1) $g$ is a continuous function from $[0,\infty)$ to $D$;\\
(2) $m$ is a continuous  $d$ dimensional function on $[0,\infty)$ with $m(0)=0$ and
each component of $m$ is non-decreasing;\\
(3) the Skorokhod equation holds, namely, $g(t)=f(t)+Rm(t)$ for all $t\ge 0$;\\
\phantom{xxx}  and\\
(4) $m_j$ increases only at those time $t$ when $g_j(t)=0$, $j=1,\dots,d.$
}
\end{definition}

The stochastic Skorokhod problem is very similar, but with the driving function
replaced by  a Brownian motion.
Let $(\Omega, \calF, \{\calF_t\},  \P)$
be a filtered probability space with $\{\calF_t\}$ being a right continuous
complete filtration.

\begin{definition}\label{skor-prob}
Let $B(t)$ be standard $d$-dimensional Brownian motion in $\R^d$ started at
$x_0\in D$ and adapted to $\{\calF_t\}$.
The stochastic Skorokhod problem is to find $X(t)$ and $M(t)$, adapted to $\{\calF_t\}$, so that
almost surely\\
(1)   $X$ is a  
continuous $d$-dimensional vector-valued process taking values in $D$;\\
(2) $M$ is a continuous $d$ dimensional vector-valued process with $M(0)=0$
and each component 
is non-decreasing;\\
(3) the Skorokhod equation holds, namely,
$X(t)=B(t)+RM(t)$  for all  $t\ge 0$;\\
\phantom{xxx}  and\\
(4)  $M_j$ increases only at those times $t$ when 
$X_j(t)=0$.
\end{definition}

Nothing precludes $\{\calF_t\}$ being strictly
larger than the filtration generated by $B$. 
We refer to $X$ as obliquely reflecting
Brownian motion (ORBM) in $D$. 
The process $M_i$ is referred to as the local time of $X$ on the $i^{th}$
face $D_i=\{x\in D: x_i=0\}$. When $X(t)\in D_i$, the direction of reflection
is given by the $i^{th}$ row of $R$.
An edge is a set of the form $D_j\cap D_k$ for some $j\ne k$.
For the cases we are interested in,
no $M_i$ charges any edge (that is, if $E$ is an edge, then 
$\int_0^t 1_E(X_s)\, dM_i(s)=0$ for all $t$ almost surely) and so
the direction of reflection will be  immaterial 
when $X(t)$ is at an edge 
as long as it is deterministic and points into the interior of $D$;
see \cite{RW}. 

ORBM is the limit process in many queueing
theory models, and there is a large literature concerning this topic. See
\cite{W95} for some
of the references. The model has also been of independent interest
in probability theory and is related to
partial differential equations with oblique boundary conditions.

 For a vector $v$, writing $v>0$ means each component
of $v$ is positive, $v\ge 0$ that each component is non-negative. 
A matrix $R$ is an $\calS$-matrix if there exists $x\ge 0$ with $Rx>0$. 
A principal submatrix of $D$ is a square matrix obtained from $R$ by deleting
some of the rows and corresponding columns of $R$. A matrix $R$ is called completely-$\calS$ if
$R$ and all of its principal submatrices are $\calS$-matrices. 
A necessary and sufficient condition  for  the deterministic Skorokhod problem 
to have a solution for all continuous  driving functions $f$
is that  $R$ be completely-$\calS$; 
see \cite{BEK}, \cite{MVdH},  and \cite{DW}. 
In Remark \ref{diagonal-ones} 
we show that, as far as questions of existence and uniqueness of solutions to the 
Skorokhod problem go, we may assume that all the diagonal elements of $R$
are equal to 1. We make that assumption throughout the remainder of the paper.

There has been considerable interest in when the solutions to the 
deterministic and stochastic
Skorokhod problems
are unique. Harrison and Reiman \cite{HR} proved that 
uniqueness holds for both problems (pathwise uniqueness in the stochastic case)
 if $Q=I-R$ is the transition matrix of a class of transient
Markov chains and Williams \cite{W95} observed that the proof actually
holds provided the spectral radius of $|Q|$ is strictly less than 1, where
$|Q|$ is the matrix whose entries are the absolute values of the corresponding
entries of $Q$.

The first main result of this paper  concerns pathwise uniqueness for
the stochastic
Skorokhod equation, and can be considered a critical case since we look at when the
spectral radius of $|Q|$ is exactly 1.    

\begin{theorem}\label{main-theorem}
Let $B$ be a Brownian motion adapted to $\{\calF_t\}$.  
Suppose $R$ is completely-$\calS$ and the spectral radius of $|Q|$ is less than or equal to 1.
If $(X,M)$ and $(X',M')$ are two solutions to the Skorokhod problem 
given by
Definition \ref{skor-prob}, then almost surely $(X(t),M(t))=(X'(t),M'(t))$
for all $t$.
Moreover there exists a solution $(X,M)$ to the Skorokhod equation that is
adapted to the filtration generated by $B$.
\end{theorem}

The proofs in \cite{HR,W95} use the contraction mapping principle and require that the spectral radius of $|Q|$ be strictly less than 1.
The proof of Theorem \ref{main-theorem} is necessarily quite different.
Our tools include a fixed point theorem of Ishikawa \cite{Ishikawa} as well as
a decomposition of $Q$ analogous to that of a finite state Markov chain.

The current authors recently showed in \cite{BaBu} that
there is a large class of $2\times 2$ matrices which are completely-$\calS$ but for which
pathwise uniqueness for the corresponding ORBM does  not hold.

The questions of  uniqueness in law  for the solution of  
the stochastic Skorokhod problem and the strong Markov property for the set of
 solutions as the starting point varies 
have been settled by \cite{TW}.

One can 
equivalently write the
Skorokhod equation in terms of a real-valued local time; see \cite[Section 2]{BaBu}.
Also for matrices $R$ that are not
completely-$\calS$, one can characterize ORBMs that correspond to $R$, but they will not
be solutions to the Skorokhod equation; see \cite{VW}. 

The proof of Theorem \ref{main-theorem} for the case where $|Q|$ is what is known as an irreducible matrix is given in Section \ref{sect-irred}.
The general case is given in Section \ref{sect-general}.

Our second main result concerns uniqueness for the deterministic Skorokhod problem 
equation in two
dimensions and consists of two theorems. Suppose $R$ has the form 
$R=\begin{pmatrix}
        1 & a_1  \\
        a_2& 1 \\
        \end{pmatrix}$.

There are five cases to consider:\\
(1) $|a_1a_2|<1$;\\
(2) $|a_1a_2|=1$, $a_1,a_2$ are of opposite signs;\\
(3) $|a_1a_2|=1$, $a_1,a_2$ are both positive;\\
(4) $|a_1a_2|>1$, $a_1,a_2$ are of opposite signs;\\
(5) $|a_1a_2|>1$, $a_1,a_2$ are both positive.\\
\noindent (It is easy to see that $R$ will not be 
completely-$\calS$ if $|a_1a_2|\ge 1$ and $ a_1,a_2$ are both negative.)

Uniqueness holds in Case (1) by \cite{HR,W95}. Mandelbaum \cite{Man} gave a brief
sketch to show uniqueness holds for Case (2). He also provided in \cite{Man}  a detailed proof
for a counterexample to show uniqueness fails in Case (4). 

Mandelbaum's paper \cite{Man} has never been published. However an exposition of his results for Cases (2) and (4) may be
found in \cite{Ba-lec-notes}. The counterexample (Case (4)) has been also presented in the
Ph.D.~thesis of Whitley \cite{Whit}. Stewart \cite{St} used similar methods
to prove a counterexample in a related problem.
See also \cite{BEK} for a counterexample to uniqueness when the dimension is 3.

Our first theorem covers Case (3), which may be considered a critical case.

\begin{theorem}\label{T1-uniq}
Suppose $a_1>0, a_2>0, a_1a_2=1$, $f$ is a continuous driving function, and $(g,f,m)$ and $(\ol g, f, \ol m)$
are two solutions to the deterministic Skorokhod problem with matrix
$R=\begin{pmatrix}
        1 & a_1  \\
        a_2& 1 \\
        \end{pmatrix}$. Then $g(t)=\ol g(t)$ for all $t$.
\end{theorem}

Although $g$ is uniquely determined, $m$ is not: see Remark \ref{R1-uniq}.
We point out that the argument we use to prove Theorem \ref{T1-uniq} is well 
known to experts in the field. 
If we allow discontinuous driving functions $f$, in particular piecewise
constant ones, then  there are cases where there will be non-uniqueness for some driving functions (\cite{Man}).

Our second theorem settles Case (5).
This result and its proof are new.
Our method was motivated by 
the proofs in \cite{BaBu}.

\begin{theorem}\label{T2-non}
Suppose  $R=\begin{pmatrix}
        1 & a_1  \\
        a_2& 1 \\
        \end{pmatrix}$
and $a_1>0, a_2>0, a_1a_2>1$. Then there is a continuous driving function $f$ for which there is more
than one solution to the deterministic Skorokhod problem.
\end{theorem}

The proofs of Theorems \ref{T1-uniq} and \ref{T2-non} are given in Section
\ref{sect-2d}.

\begin{remark}\label{diagonal-ones}
If $R$ is completely-$\calS$, then all the diagonal elements must be positive. If we let $\wt m_i(t)=R_{ii}m_i(t)$ and $\wt R_{ij}=R_{ij}/R_{jj}$
for $i,j=1, \ldots, d$, then the Skorokhod equation can be rewritten as
$g(t)=f(t)+\wt R \wt m(t)$. Note $\wt R_{ii}=1$ for each $i$. 
Therefore there is no loss of generality in assuming that the diagonal elements of $R$ are all equal to 1. The same remark holds for the 
stochastic Skorokhod problem.
\end{remark}

\medskip

\noindent \textbf{Acknowledgments.} We thank A.~Mandelbaum for providing
us with a copy of his unpublished paper. We also thank R.~Williams for numerous helpful discussions on the subject of this paper.

\section{The irreducible case}\label{sect-irred}

A $d\times d$ matrix $A$ is non-negative if each entry of $A$ is non-negative,
and we write $A\ge B$ if $A-B$ is non-negative. We denote the spectral radius
of $A$ by $\rho(A)$; this is the maximum of the moduli  of the eigenvalues of
$A$. Given a matrix $A$, we denote by $|A|$ the matrix  whose entries are
the absolute values of the corresponding entries of $A$, and similarly $A^+$ 
for the matrix whose entries are the positive parts of the corresponding entries
of $A$. We write $v>0$ for a $d$-dimensional vector $v$ if all the coordinates
are positive, $v\ge 0$ if all the coordinates are non-negative. We use $|v|$ for 
$(\sum_{i=1}^d |v_i|^2)^{1/2}$ and $\norm{A}=\sup\{|Av|: |v|\le 1\}$. Finally we use $A^T$ for the transpose of a matrix $A$.

A non-negative matrix $A$ is irreducible if for each $i,j$ there exist  $k_0, k_1, 
\ldots, k_n$ with $k_0=i$ and $k_n=j$ such that $A_{k_m,k_{m+1}}>0$ for $0\le m\le n-1$. The $k$'s and
$n$ can depend on $i$ and $j$. A number of equivalent definitions of 
irreducibility can be found in \cite[Theorem 6.2.24]{HJ}. If the row sums of 
$A$ are each equal to 1, then $A$ can be viewed as the transition matrix of
a Markov chain, and irreducibility of $A$ is the same as irreducibility in the
Markov chain context.

We need a few well-known facts about non-negative matrices.

\noindent (1) If $A\ge B\ge 0$, then 
 (\cite[Theorem 8.1.18]{HJ})
\begin{equation}\label{irred-bf1}
\rho(A)\ge \rho(B).
\end{equation}

\noindent (2) The Gelfand formula 
(\cite[Corollary 5.6.14]{HJ}):
\begin{equation}\label{irred-bf2}
\rho(A)=\lim_{m\to \infty}\norm{A^m}^{1/m}.
\end{equation}

\noindent (3) The Perron-Frobenius theorem
(\cite[Theorem 8.4.4]{HJ}): If $A\ge 0$ is irreducible, then
there exists a positive eigenvalue $r=\rho(A)$ and left and right positive
eigenvectors $y,x$ such that $y^TA=ry^T$ and $Ax=rx$.

\noindent (4) A corollary to the  Perron-Frobenius theorem
(\cite[Theorem 8.3.1]{HJ}): If $A\ge 0$, then
there exists a non-negative non-zero $x$
such that 
$Ax=\rho(A)x$.

Throughout the remainder of this section we assume $Q=I-R$ and $|Q|$
is irreducible with $\rho(|Q|)\le 1$.
Note that this implies that the diagonal elements of $Q$ are zero.

We need the following elementary lemma.

\begin{lemma} \label{add-identity}
If $C\ge 0$ and $\beta\ge0$, then 
$\rho(C+\beta I)=\rho(C)+\beta$.
\end{lemma}

\begin{proof} By the corollary to the Perron-Frobenius theorem described in (4) above,  there exists an eigenvector
$x$ for $C$ such that $x\ge 0$, $x\ne 0$, and $Cx=\rho(C)x$. Then $(C+\beta I)x=\rho(C)x+\beta x$,
which implies that $C+\beta I$ has an eigenvalue of size 
$\rho(C)+\beta$, and consequently $\rho(C+\beta I)
\ge \rho(C)+\beta$. 

On the other hand,
$C+\beta I$ will also be non-negative, hence there exists an
eigenvector $\wh x$ for $C+\beta I$ such that $\wh x\ge 0$, $\wh x\ne 0$, 
and $(C+\beta I)\wh x=\rho(C+\beta I)\wh x$. Then 
$$C\wh x=(C+\beta I)\wh x-\beta\wh x=\rho(C+\beta I)
\wh x-\beta \wh x,$$
which implies that $C$ has an eigenvalue of size 
 $\rho(C+\beta I)-\beta$, hence $\rho(C)\ge \rho(C+\beta I)-\beta$.
\end{proof}

In the case $Q\ge 0$ we have the following.       

\begin{proposition}\label{irred-p2}
Suppose $Q\ge 0$, $R$ is completely-$\calS$, and $\rho(Q)\le 1$. Then $\rho(Q)<1$.
\end{proposition}

\begin{proof}
Recall that we are assuming that $|Q|$ is irreducible in this section.
Suppose $\rho(Q)=1$. By the Perron-Frobenius theorem there exists
a left eigenvector $y>0$ such that $y^TQ=y^T$. Then
$y^TR=0$. Since $R$ is completely-$\calS$, there exists $x\ge 0$ such that
$Rx>0$. Now consider $y^TRx$. On the one hand this equals 0. On the other,
because $y>0$ and $Rx>0$, then $y^TRx>0$, a contradiction.
\end{proof}

\begin{lemma}\label{irred-L1}
Suppose $A\ge B\ge 0$, $A\ne B$, and $A$ is irreducible. Then $\rho(A)>\rho(B)$.
\end{lemma}

\begin{proof}
Since $A\ge B$ but $A\ne B$, there exist $i_0,j_0$ such that $A_{i_0j_0}
> B_{i_0j_0}$. Fix $i\in \{1, \ldots, d\}$ for the moment. There exists $n_1(i)$
such that $A^{n_1(i)}_{ii_0}>0$ and $n_2(i)$ such that $A^{n_2(i)}_{j_0i}>0$. 
Set $n(i)=n_1(i)+n_2(i)+1$. 
Therefore
\begin{equation}\label{irr-E1}
A^{n_1(i)}_{ii_0}A_{i_0j_0}A^{n_2(i)}_{j_0i} 
>B^{n_1(i)}_{ii_0}B_{i_0j_0}B^{n_2(i)}_{j_0i}.
\end{equation}  
Since
$$A^{n(i)}_{ii}=\sum A^{n_1(i)}_{ik_1}A_{k_1k_2} A^{n_2(i)}_{k_2i},$$
where the sum is over all indices $k_1, k_2 
\in \{1, \ldots, d\}$  and a 
similar expression holds for $B^{n(i)}_{ii}$, using \eqref{irr-E1} we conclude
$A^{n(i)}_{ii}>B^{n(i)}_{ii}$. 

Suppose  $K$ is a positive integer. Then
$(A^{n(i)}_{ii})^K > (B^{n(i)}_{ii})^K$.
Observe that
\begin{align}\label{j16.1}
A^{Kn(i)}_{ii}=\sum A^{n(i)}_{ik_1}A^{n(i)}_{k_1k_2}\cdots 
A^{n(i)}_{k_{K-2},k_{K-1}} A^{n(i)}_{k_{K-1}i}
\end{align}
where the sum is over all $k_1, \ldots, k_{K-1}\in \{1, \ldots, d\}$
and a similar expression holds for $B^{Kn(i)}_{ii}$.
Each summand in \eqref{j16.1} is greater than or equal to the corresponding   summand for $B$. However the summand where we have $k_1=k_2=…=k_{K-1}=i$ is equal to $(A^{(n(i)}_{ii})^K$, and this is strictly larger than  $(B^{(n(i)}_{ii})^K$. 
It follows that
$A^{Kn(i)}_{ii}> B^{Kn(i)}_{ii}$.

Let $N$ be the least common multiple of $\{n(1), \ldots, n(d)\}$. Then
$A^N_{ii}>B^N_{ii}$ for each $i$, and hence there exists $\eps>0$ such that 
$$A^N\ge B^N+ \eps I.$$
Therefore $\rho(A^N)\ge \rho(B^N+\eps I)= \rho(B^N)+\eps$ by \eqref{irred-bf1} and
Lemma \ref{add-identity}.

Finally, by \eqref{irred-bf2},
$$\rho(A^N)=\lim_{m\to \infty} \norm{A^{Nm}}^{1/m}
=\lim_{m\to \infty}( \norm{A^{Nm}}^{1/Nm})^N=\rho(A)^N,$$
and similarly for $B$. Our result now follows.
\end{proof}

Fix $t_0>0$.
Let $\calC_d[0,t_0]$ be the set of functions $f$ mapping $[0,t_0]$ to $\R^d$ such that each component of $f$ is continuous.
Let $\calC^{*}_d[0,t_0]$ be the set of functions $f$ in $\calC_d[0,t_0]$
such that each component of $f$ is non-decreasing.

Let $f$ be a fixed element of $\calC_{d}[0,t_0]$ with $f(0)\ge 0$.
We define the non-linear operator $T$ on $\calC_d[0,t_0]$ by
\begin{equation}\label{def-T}
Tg(t)=\sup_{s\le t} \Big(Qg(s)-f(s)\Big)^+.
\end{equation}
We define the non-linear operator $U$ on $\calC_d[0,t_0]$ by
\begin{equation}\label{def-U}
Ug(t)=\tfrac12 g(t)+\tfrac12 Tg(t).
\end{equation}
We denote the iterates of $U$ by $U^2g=U(Ug)$, etc.

\begin{theorem}\label{irred-T1}
Suppose  $R$ is completely-$\calS$.
Let $Q=I-R$. 
Suppose either (a) $Q=0$ or (b) $|Q|$ is irreducible and $\rho(|Q|)\le 1$. 

\noindent (i) There exists a constant $C_1$ such that if $h\in \calC_d^*[0,t_0]$ and $h(0)=0$, then
\begin{align}\label{j23.1}
|Th(t)-Th(s)|\le C_1\Big(|h(t)-h(s)|+\sup_{s\le u_1<u_2\le t}
|f(u_2)-f(u_1)|\Big).
\end{align}

\noindent(ii) Suppose $g(0)=0$ and $g\in \calC_d^*[0,t_0]$.  For all $n\ge 1$
we have $U^ng\in \calC_d^*[0,t_0]$ and $U^ng(0)=0$. Moreover
there exists a constant $C_2$ not depending on $n$ such that for all
$0\le s<t\le t_0$ and all $n\ge 1$  we have
\begin{align*}
|U^ng(t)-U^ng(s)|\le C_2\Big(|g(t)-g(s)|+\sup_{s\le u_1<u_2\le t}
|f(u_2)-f(u_1)|\Big).
\end{align*}

\end{theorem}

\begin{proof}

(i) 
Suppose $h(0)=0$ and $h\in \calC_d^*[0,t_0]$.
Let 
\begin{align*}
F_i(s,t)=\sup_{s\le u_1<u_2\le t} |f_i(u_2)-f_i(u_1)|.
\end{align*}
Writing the formula for $Th$ in terms of coordinates, we have
\begin{equation}\label{T-coord}
(Th)_j(t)=\sup_{r\le t} \Big(\sum_{k=1}^d Q_{jk}h_k(r) -f_j(r)\Big)^+.
\end{equation}
If $w$ is a non-negative function and $0\leq s \leq t$, then
$$\sup_{r\le t} w(r)\le \sup_{r\le s} w(r)+\sup_{s\le r\le t}(w(r)-w(s)).$$
This, \eqref{T-coord} and the inequality $(a+b)^+\le a^+ + b^+$ imply that
$$(Th)_j(t)-(Th)_j(s)\le \sup_{s\le r\le t}(Q(h_k(r)-h_k(s)))_j+F_j(s,t).$$
All the diagonal entries of $Q$ are 0. 
If  $Q_{jk}<0$, then  $r\to Q_{jk}h_{k}(r)$ 
is non-increasing. If  $Q_{jk}\geq 0$, then
$\sup_{s\le r\le t}(Q(h_k(r)-h_k(s)))_j
= (Q(h_k(t)-h_k(s)))_j$.
Therefore
\begin{equation}\label{Th-bound}
(Th)_j(t)-(Th)_j(s)\le \sum_{k=1}^d Q^+_{jk}(h_k(t)-h_k(s))+F_j(s,t).
\end{equation}
 This implies \eqref{j23.1} because $Q^+$ is a fixed matrix not depending on $h$.

(ii) It is easy to see that for each $n$, we have $U^ng(0)=0$, $U^ng$ is non-decreasing, and $U^ng$ is continuous.

In matrix terms, \eqref{Th-bound} can be written
\begin{equation}\label{T-diff}
Th(t)-Th(s)\le Q^+(h(t)-h(s))+F(s,t).
\end{equation}
Let
$$P=\tfrac12 Q^+ + \tfrac12 I.$$
Since $h\in \calC^*_d[0,t_0]$,
\begin{equation}\label{U-diff}
Uh(t)-Uh(s)\le P(h(t)-h(s))+\tfrac12 F(s,t).
\end{equation}

Applying this to $h=U^ng$ we obtain
$$U^{n+1}g(t)-U^{n+1}g(s)\le P(U^ng(t)-U^ng(s))+\tfrac12 F(s,t).$$
An induction argument shows that
\begin{equation}\label{T-matrix}
U^ng(t)-U^ng(s)\le P^n(g(t)-g(s))+\tfrac12 \sum_{i=1}^n P^{i-1}F(s,t).
\end{equation}

In the case that $Q\ge 0$, we have by Proposition \ref{irred-p2} that
$\rho(Q^+)=\rho(Q)<1$.
In the case where $Q$ has at least one negative entry we use
Lemma \ref{irred-L1} with $A=|Q|$ and $B=Q^+$ to see that $\rho(Q^+)<
\rho(|Q|)\le 1$.
By Lemma \ref{add-identity} we have $\rho(P)<1$.
By \eqref{irred-bf2} there exists $M$ such that $\norm{P^M}<1$.
For any $n$ we can write $n=kM+j$, where $k\ge 0$ and $0\le j<M$. Let
$\kappa = \norm{P^M}$ and $\lambda=\norm{P}$. 
Then \eqref{T-matrix} yields
\begin{equation}\label{T-equicontinuity}
U^ng(t)-U^ng(s)\le \kappa^k \lambda^j|g(t)-g(s)|
+\tfrac12 \Big(\sum_{k=0}^\infty \kappa^k\Big) \Big(\sum_{j=0}^{M-1} \lambda^j\Big)\, |F(s,t)|.
\end{equation}
The sum over $k$ is finite because $\kappa<1$. It may be that $\lambda>1$, but the
sum over $j$ is a finite sum. The desired uniform bound on $U^ng(t)-U^ng(s)$
thus follows from \eqref{T-equicontinuity}.
\end{proof}
 
Next we define another norm on $\calC_d[0,t_0]$. Suppose first that
$Q\ne 0$. Since $|Q|$ is non-negative and irreducible,
by the Perron-Frobenius theorem there exists $y>0$ such that $y^T|Q|=ry^T$,
where $0<r=\rho(|Q|)\le 1$.

Define
\begin{equation}\label{def-norm}
|f|_*=\sup_{s\le t_0}\sum_{i=1}^d y_i|f_i(s)|.
\end{equation}
Since $y_i>0$ for each $i$, this norm is equivalent to  the norm 
$$|f|=\sup_{s\le t_0}\sum_{i=1}^d |f_i(s)|.$$
If $Q=0$, let $|f|_*=\sup_{s\le t_0} |f(s)|$.

We show $T$ is a non-expansive map with respect to $|\,\cdot\,|_*$.

\begin{proposition}\label{irred-P3}
Suppose $g\in \calC_d[0,t_0]$, either (i) $Q=0$ or (ii) $|Q|$ is irreducible, and
$r=\rho(|Q|)\le 1$.
Then $T$ is a non-expansive map:
if $g_1,g_2\in \calC_d[0,t_0]$, then
$$|Tg_1-Tg_2|_*\le |g_1-g_2|_*.$$
\end{proposition}

\begin{proof}

Suppose $Q\ne 0$. 
If $g\in \calC_d[0,t_0]$, then 
\begin{align*}
|Qg|_*&= \sup_{s\le t_0} \sum_{i,j=1}^d y_i|Q_{ij}\,g_j(s)|
\le \sup_{s\le t_0} \sum_{j=1}^d (y^T|Q|)_j |g_j(s)|\\
&=\sup_{s\le t_0} \sum_{j=1}^d ry_j|g_j(s)|\le |g|_*
\end{align*}
since $r\le 1$. Then
$$|Qg_1-Qg_2|_*=|Q(g_1-g_2)|_*\le |g_1-g_2|_*.$$
It follows readily from this and the definition of $T$ that $T$ is 
non-expansive.

In the case $Q=0$, we see that $|Qg_1-Qg_2|_*=0$ and again $T$ is non-expansive.
\end{proof}

A key result of this section is the following.

\begin{theorem}\label{irred-T2}
Suppose $Q=I-R$ and either (i) $Q=0$ or (ii) $|Q|$ is irreducible with
$\rho(|Q|)\le 1$.  Suppose $f\in \calC_d[0,t_0]$ with $f(0)\ge 0$.
Let $g^{(0)}$ be the identically 0 function.
For $k\ge 0$ define
$$g^{(k+1)}=\tfrac12\Big((g^{(k)}+T(g^{(k)})\Big).$$
Then $g^{(k)}$ converges with respect to $|\,\cdot\,|_*$ to a function
$g^{(\infty)}\in \calC_d^*[0,t_0]$. Moreover $g^{(\infty)}$ is a fixed point
for $T$, that is, $Tg^{(\infty)}=g^{(\infty)}$.
\end{theorem}

\begin{proof}
Since the norms $|\,\cdot\,|$ and $|\,\cdot\,|_*$ are equivalent, sets that are compact
with respect to $|\,\cdot\,|$ are compact with respect to $|\,\cdot\,|_*$ and
vice versa.
Let $\calD$ be the set of functions $h\in\calC_d^*[0,t_0]$ 
such that $h(0)=0$ and 
\begin{align*}
|h(t)-h(s)|\le C_2\Big(\sup_{s\le u_1<u_2\le t}
|f(u_2)-f(u_1)|\Big),
\end{align*}
whenever $0\le s\le t\le t_0$,
where $C_2$ is the constant in Theorem \ref{irred-T1} (ii). 
By the Ascoli-Arzel\`a theorem,  $\calD$ is a compact subset of $\calC_d[0,t_0]$. 

Note that $g^{(k+1)}=Ug^{(k)}= U^k g^{0}$.
Using that $g^{(0)}$ is the zero function, 
Theorem \ref{irred-T1}(ii) tells us that 
$\{g^{(k)}: k\ge 0\}$ is contained in $\calD$.          
We observed in Proposition \ref{irred-P3} that $T$ is non-expansive
with respect to the norm 
$|\,\cdot\,|_*$.
By Theorem \ref{irred-T1}(i) and the Ascoli-Arzel\`a theorem, 
$T$ maps $\calD$ into a compact subset of 
$\calC_d[0,t_0]$.
We now apply a  theorem of Ishikawa (\cite[Theorem 1]{Ishikawa}) to  
obtain our result:
we take $t_n=\tfrac12$ for each $n$ in his theorem and his $D$ is our
$\calD$.
\end{proof}

Here is the result we will need in the next section.
We do not assume here that $Y$ is a Brownian motion.
We do require it to be a semimartingale so that $X$ will be a semimartingale.

\begin{theorem}\label{irred-T3}
Let $Y(t)$ be a semimartingale with trajectories in $\calC_d[0,t_0]$, adapted to a filtration
$\{\calF_t\}$, with $Y(0)\ge 0$.
Suppose either (i) $Q=0$ or (ii) $|Q|$ is irreducible and $\rho(|Q|)\le 1$. Then there exist processes
$X(t)$ taking values in $\calC_d[0,t_0]$
and $M(t)$ taking values in $\calC^*_d[0,t_0]$ such that $X(t)\ge 0$ and
$M(t)\ge 0$ for all $t$, both $X$ and $M$ are adapted to $\{\calF_t\}$,  
$M(0)=0$, 
$$X(t)=Y(t)+RM(t),$$
and for each $i$, $M_i(t)$ increases only at times when $X(t)\in D_i$, the
$i^{th}$ face of $D$.
\end{theorem}

\begin{proof} 
Let $M^{(0)}(t)$ be identically 0 and for $k\ge 0$ define
$$M^{(k+1)}(t)(\omega)=\tfrac12\Big(M^{(k)}(t)(\omega)+(TM^{(k)})(t)(\omega)\Big),$$
where in the definition of $T$ given by \eqref{def-T} we replace $f(s)$ by
$Y(s)(\omega)$. Induction shows that each $M^{(k)}$ is adapted to 
$\{\calF_t\}$ and its components are non-decreasing. 

Since $Y(s)(\omega)$ is continuous almost surely, we apply Theorem 
\ref{irred-T2} for each $\omega$. Let $M(t)(\omega)$ be the limit 
of $M^k(t)(\omega)$ as $k\to \infty$.
Clearly $M(t)$ is also adapted to $\{\calF_t\}$ and its components are non-decreasing.

Let $$X(t)=Y(t)+RM(t).$$
The only parts of the theorem that are not easy consequences of Theorem
\ref{irred-T2} are that $X(t)\ge 0$ and that $M_i(t)$ increases only when
$X(t)\in D_i$.

Since $M$ is a fixed point for $T$, for each $\omega$ we have
\begin{align}
M(t)&=\sup_{s\le t}\Big[QM(s)-Y(s)\Big]^+=\sup_{s\le t}
\Big[-RM(s)+M(s)-Y_s\Big]^+\label{irred-T3E1}\\
&=\sup_{s\le t} \Big[M(s)-X(s)\Big]^+,\notag
\end{align}
using the definition of $X$.
Therefore $M(t)\ge M(t)-X(t)$, which implies $X(t)\ge 0$.

Suppose $M_i$ has a point of increase at a time $t_1$
yet $X_i(t_1)>0$; we will obtain a contradiction.
Let us suppose $t_1\in (0, t_0)$, the cases when $t_1=0$ or $t_1=t_0$ being similar.
By the continuity of $X_i$ there exist $h,\eps>0$ such that $t_1+h\le t_0$, $t_1-h\ge 0$,  
and $X_i(s)\ge \eps$ for $t_1-h\le s\le t_1+h$. Since $t_1$ is a point of
increase for $M_i$, we can take $\eps$ smaller if necessary so that
$M_i(t_1-h)\le M_i(t_1+h)-\eps$. 
If $s\le t_1-h$, then
$$M_i(s)-X_i(s)\le M_i(s)\le M_i(t_1-h)\le M_i(t_1+h)-\eps.$$
If $t_1-h\le s\le t_1+h$,
$$M_i(s)-X_i(s)\le M_i(s)-\eps\le M_i(t_1+h)-\eps.$$
But then 
$$\sup_{s\le t_1+h} (M_i(s)-X_i(s))\le M_i(t_1+h)-\eps,$$
which contradicts
\eqref{irred-T3E1} with $t$ replaced by $t_1+h$.
\end{proof}

\section{Pathwise uniqueness}\label{sect-general}

The reader who is only interested in the irreducible case may at this point 
jump to the proof of Theorem \ref{main-theorem} at the end of this section.

In this section we consider the case where $Q$ might not be irreducible, and
reduce it to the irreducible case by a decomposition of $R$.

Let $1\le J\le d$. Let $I_J$ be the $J\times J$ identity matrix.
We start with an elementary lemma.

\begin{lemma}\label{general-L1}
Suppose $A\ge 0$ and $A_0$ is a principal submatrix of $A$. Then 
$\rho(A_0)\le \rho(A)$.
\end{lemma}

\begin{proof}
$A_0$ is obtained from $A$ be deleting certain rows and columns, say the $i_1,
\ldots, i_k$ rows and columns. Let $B$ be the matrix obtained from $A$ by 
changing all entries in the $i_1,\ldots, i_k$ rows and 
columns to 0's. By \eqref{irred-bf1},
$\rho(A)\ge \rho(B)$.
Now delete those rows and columns from $B$ to obtain $A_0$. It is easy to check that
the non-zero eigenvalues of $A_0$ and $B$ are the same.
\end{proof}

We say that $\calP(J)$ holds if the following occurs.

\begin{definition} \label{Prop-PJ}
{$\calP(J)$: Let $R'$ be any principal submatrix of $R$ of
size $J\times J$.
Let $Q'=I_J-R'$. For any  semimartingale $Y'$ taking values
in $\calC_J[0,t_0]$ with $Y_0\ge 0$ and measurable with respect to
a filtration $\{\calF_t\}$ there exist $J$-dimensional processes
$X'$ and $M'$ with trajectories in $\calC_J[0,t_0]$
and adapted to $\{\calF_t\}$, $X'(t)\ge 0$ for all $t$,
$M'(0)=0$, each component of $M'$ is non-decreasing and increases only
when the corresponding component of $X'$ is 0, and
$$X'(t)=Y'(t)+R'M'(t).$$} 
\end{definition}

Note that for $\calP(J)$ to hold, the condition stated in 
Definition \ref{Prop-PJ} has to hold for each continuous semimartingale $Y'$,
not just for a particular one.

We will prove that Property $\calP(J)$ holds for each $J\le d$ by induction, but
first we must show how to decompose the set of indices into communicating classes. This decomposition is analogous to decomposing a Markov chain state space into communicating classes.

Let $A$ be a $J\times J$ matrix. 
We say $i\to j$  if there exist $i_0,i_1, \ldots, i_k$ such that 
$i_0=i$, $i_k=j$, and $A_{i_mi_{m+1}}\ne 0$ for $m=0, \ldots, k-1$.
We say $i$ and $j$ are equivalent if $i=j$ or if $i\to j$ and $j\to i$. 
If an equivalence class has more than one element, let $i$ and $j$ be any two elements;  then
$i\to j$ and $j\to i$, so $i\to i$ and $j\to j$. It follows from this that an equivalence class with at least two elements is irreducible in the matrix sense.
Note also that a square matrix is irreducible if and only if its transpose
is irreducible.

Let $E_1, \ldots, E_m$ be the equivalence classes.
Let us write $E_k\Rightarrow E_\ell$ if there exists
$i\in E_k$ and $j\in E_\ell$ such that $i\to j$. We cannot have both
$E_k\Rightarrow E_\ell$ and $E_\ell\Rightarrow E_k$ for $\ell\ne k$ or else
$E_k$ and $E_\ell$ would not be distinct equivalence classes.

\begin{proposition}\label{general-P1}
Suppose $R$ is completely-$\calS$, $Q=I_d-R$, and $\rho(|Q|)\le 1$. Then
Property $\calP(J)$ holds for $J=d$.
\end{proposition}

\begin{proof}
We show Property $\calP(J)$ holds for each $J\le d$ using induction. The case when
$J=1$ follows by Theorem \ref{irred-T3}.
We suppose $J_0\ge 2$ and that Property $\calP(J)$ holds for $J<J_0$, and we will
prove that Property $\calP(J)$ holds for $J=J_0$.

Let $R'$ be a $J_0\times J_0$ principal submatrix of $R$, and let
$Q'=I_{J_0}-R'$. 
Since $R’$ is a principal submatrix of $R$, $Q’$ is a principal submatrix of $Q$,  and, therefore,
$|Q’|$ is a principal submatrix of $|Q|$. We assumed that $\rho(|Q|)\le 1$ so Lemma \ref{general-L1} implies that $\rho(|Q'|)\le 1$.
If  $|Q'|$ is irreducible, then
we apply Theorem \ref{irred-T3} and we have that Property $\calP(J)$ holds
and we are done. 

So we suppose $|Q'|$ is not irreducible.
Let $A=|Q'|^T$. We will decompose 
$A$ 
rather than $|Q'|$ because if $R_{ij}\ne 0$, then
the local time $M_j$ of the component $X_j$ gives a push to the component $X_i$; we want the
notation $i\to j$ to be consistent with this.

Decompose $A$ into equivalence classes as described above. There must be 
at least one $E_{i_0}$ such that $E_k\not\Rightarrow E_{i_0}$ for all $k\ne {i_0}$. 
By renumbering the axes we may suppose $E_{i_0}=\{1, \ldots, j_1\}$ for
some $1\le j_1< J_0$. 
If $j\in E_{i_0}$ and $k\notin E_{i_0}$, then $k\not\to j$, so
\begin{equation}\label{decompose} |Q'|_{jk}=A^T_{kj}=0.
\end{equation}

Let $Y'$ be a $J_0$-dimensional semimartingale adapted to $\{\calF_t\}$ with $Y'(0)\ge 0$.
Let $R^{(a)}$ be the $j_1\times j_1$ principal submatrix of $R'$ obtained by deleting the
$j_1+1, \ldots, J_0$ rows and columns of $R'$, and let $Q^{(a)}=I_{j_1}-R^{(a)}$.
By our construction, either $Q^{(a)}$ will be a $1\times 1$ matrix or $|Q^{(a)}|$
will be irreducible. Since $R$ is completely-$\calS$, so is $R^{(a)}$.
By Lemma \ref{general-L1},  
$\rho(|Q^{(a)}|)\le 1$. 
Let $Y^{(a)}$ be the $j_1$-dimensional vector obtained from
 $Y'$ by deleting  the $j_1+1, \ldots, J_0$ entries.
Using Theorem \ref{irred-T3} 
we can find $X^{(a)}$ and $M^{(a)}$ satisfying the 
Skorokhod problem 
for $R^{(a)}$ and $Y^{(a)}$.
Also both $X^{(a)}$ and $M^{(a)}$ are adapted to $\{\calF_t\}$.
Our solution $(X^{(a)},M^{(a)})$ satisfies
\begin{equation}\label{general-E10}
X_j^{(a)}(t)=Y_j^{(a)}(t)+\sum_{k=1}^{j_1} R^{(a)}_{jk}M^{(a)}_k(t), \qquad j=1, \ldots, j_1.
\end{equation}

Let $ R^{(b)}$ be the principal submatrix of $R'$ obtained by deleting the
$1, \ldots, j_1$ rows and columns, let $Q^{(b)} =I_{J_0-j_1}-R^{(b)}$, and 
note as above $\rho(|Q^{(b)}|)\le 1$.
Let $Y^{(b)}$ be the $(J_0-j_1)$-dimensional vector-valued semimartingale
defined by 
\begin{equation}\label{general-E21}
Y^{(b)}_{j}(t)=Y'_{j+j_1}(t)+\sum_{k=1}^{j_1}R'_{j+j_1,k}M^{(a)}_{k}(t), \qquad j=1, \ldots, J_0-j_1.
\end{equation}

The process $Y^{(b)}$ is adapted to $\{\calF_t\}$.
Using our induction hypothesis we can find
$(X^{(b)},M^{(b)})$ adapted to $\{\calF_t\}$ satisfying the Skorokhod problem, and in particular,
\begin{equation}\label{general-E12}
X_{j}^{(b)}(t)=Y_{j}^{(b)}(t)+\sum_{k=1}^{J_0-j_1} R^{(b)}_{jk}M^{(b)}_k(t), \qquad j=1, \ldots,J_0- j_1.
\end{equation}

We now set $X'_j(t)$ equal to $X^{(a)}_j(t)$ if $j\le j_1$ and equal to
$ X^{(b)}_{j-j_1}$ if $j>j_1$ and similarly for $M'$.

If $j\le j_1$ and $k>j_1$, then $j\in E_{i_0}$ and $k\notin E_{i_0}$ and
by \eqref{decompose}, $R'_{jk}=Q_{jk}=0$.
Using this and \eqref{general-E10} we obtain
\begin{equation}\label{general-E31}
X'_j(t)=Y'_j(t)+\sum_{k=1}^{J_0} R'_{jk}M'_k(t).
\end{equation}
If $j>j_1$, using
\eqref{general-E21} and \eqref{general-E12}, we see that
$X^{(b)}$ solves  
\begin{equation}\label{general-E13}
X^{(b)}_{j-j_1}(t)
=Y'_{j}(t)+\sum_{k=1}^{j_1} R^{(a)}_{jk}M_k^{(a)}(t)
+\sum_{k=1}^{J_0-j_1}R^{(b)}_{j-j_1,k}M_k^{(b)}(t).
\end{equation}
Then
\begin{equation}
X'_j(t)=X^{(b)}_{j-j_1}(t)=Y'_j(t)+\sum_{k=1}^{j_1} R'_{jk}M'_k(t)
+\sum_{k=j_1+1}^{J_0} R'_{jk} M'_k(t).
\end{equation}
Therefore \eqref{general-E31} is also valid if $j>j_1$.

Thus
Property $\calP(J_0)$ holds.
By induction, Property $\calP(J)$ holds for all $J$, and in particular, for $J=d$.
\end{proof}

If we now let $Y=B$, a $d$-dimensional Brownian motion started at $x_0\in D$,
and let $\{\calF_t\}$ be the filtration generated by $B$, 
we have constructed a solution $X$ to the Skorokhod problem which is
adapted to the filtration $\{\calF_t\}$. 
A solution adapted to the filtration generated by $B$ is called a strong
solution. 
The solution whose existence is guaranteed by \cite{BEK} is not in general a strong solution.

\begin{proof}[Proof of Theorem \ref{main-theorem}]
Let $B$ be a $d$-dimensional Brownian motion and let $(X,M)$ be the solution
constructed using Proposition \ref{general-P1}. 
Note that our construction shows that $(X,M)$ is adapted to the filtration
generated by $B$.
Let $(Z,N)$ be any other solution satisfying Definition \ref{skor-prob}. 
By \cite[Thm. 1.3]{TW} there is uniqueness in law for the
solution to the Skorokhod problem for a given starting point and $R$ (this
is also referred to as weak uniqueness), and 
consequently the
law of $(B,X,M)$ is equal to the law of $(B,Z,N)$. 

 Let $r$ be a non-negative rational.
Since $X(r)$ is adapted to the filtration generated
by $\{B(s): s\le r\}$, there is a Borel measurable map $\varphi$  from 
${\calC_d[0,r]}$ to $D$ such that $X(r)=\varphi(B)$ a.s. Because the 
laws of $(X,B)$ and $(Z,B)$ are equal, we must have that $Z(r)$ also equals 
$\varphi(B)$ a.s. Therefore $X(r)=Z(r)$ a.s. An analogous argument shows that $M(r)=N(r)$, a.s.

This holds for every non-negative rational. Since $X$ and $Z$ are both 
continuous a.s., then $X$ and $Z$ are identical with probability one. Similarly, $M$ and $N$ are indistinguishable.
\end{proof}

\begin{remark}\label{se-and-wu} The preceding proof 
that strong existence plus weak uniqueness implies pathwise uniqueness is the same as a well-known proof for
the analogous result for stochastic differential equations. 
\end{remark}

\section{The two dimensional case}\label{sect-2d}

In this section we consider only the deterministic Skorokhod problem in the case $d=2$.

\begin{lemma}\label{L1-uniq} Suppose $C>0$. There is a unique solution for every continuous
driving function for the deterministic Skorokhod problem with matrix $R=\begin{pmatrix}
        1 & a_1  \\
        a_2& 1 \\
        \end{pmatrix}$
if and only if there is a unique solution for every continuous driving function
for the Skorokhod problem with matrix $S=\begin{pmatrix}
        1 & Ca_1  \\
        a_2/C& 1 \\
        \end{pmatrix}$.
\end{lemma}

\begin{proof}
 If we write out $g=f+ Rm$ in coordinates and multiply the second equation by 
$1/C$, we get
\begin{align*}
g_1&=f_1+ m_1+ a_{1}m_2,\\
\frac1{C}g_2&=\frac1{C}f_2+\frac1{C} a_{2}m_1+\frac1{C} m_2.
\end{align*}
Let $\wt m_1=m_1$, $\wt g_1=g_1$, $\wt f_1=f_1$, and
$$\wt m_2=\frac1{C}m_2, \quad \wt g_2=\frac1{C}g_2, \quad \wt f_2=\frac1{C}f_2.$$
We then have 
$$\wt g=\wt f+ S \wt m.$$

It follows that there will be two distinct solutions to the Skorokhod problem for a driving
function $f$ with respect to the matrix $R$ if and only if there are  two distinct solutions
to the Skorokod problem with driving function $\wt f$ with respect to the
matrix $S$.
\end{proof}

\begin{proof}[Proof of Theorem \ref{T1-uniq}]
In view of Lemma \ref{L1-uniq}, if $a_1>0, a_2>0$ and $a_1a_2=1$, we may
reduce the question of uniqueness for the Skorokhod problem to the case where
$a_1=a_2=1$.

Let $u_1=\tfrac1{\sqrt 2}(1,1)$ and $u_2=\tfrac1{\sqrt 2}(1,-1)$. Let $(f_{u_1},f_{u_2})$ be the coordinates of $f$ with respect to the orthonormal basis
$(u_1,u_2)$. Thus $f_{u_j}$ is the inner product of $(f_1,f_2)$ with $u_j$, $j=1,2$. 
Define $g_{u_j}$ and $\ol g_{u_j}$, $j=1,2$, similarly. We use $(f_1,f_2)$, $(g_1,g_2)$, $(\ol g_1, \ol g_2)$, $(m_1,m_2)$,
and $(\ol m_1,\ol m_2)$ for the coordinates with respect to the usual basis $\{(1,0), (0,1)\}$.

Notice $Rm=(m_1+m_2, m_1+m_2)$.
A calculation shows that $g=f+Rm$  implies 
\begin{align}
g_{u_1}(t)&=f_{u_1}(t)+{\sqrt 2} (m_1(t)+m_2(t)),\label{E1-uniq}\\
g_{u_2}(t)&=f_{u_2}(t),\notag
\end{align}
and similarly with $g$ replaced by $\ol g$ and $m$ replaced by $\ol m$. 
Therefore $g_{u_2}(t)=f_{u_2}(t)=\ol g_{u_2}(t)$ for all $t$.

Let $v(t)=|g_{u_1}(t)-\ol g_{u_1}(t)|$. 
The function $v$ is continuous and $v(0)=0$. 
Let $t_0>0$ and suppose $g_{u_1}(t_0)< \ol g_{u_1}
(t_0)$. 
With respect to the coordinate
system $(u_1,u_2)$ the quadrant $D$ is a wedge symmetric about the $u_1$
axis. Therefore regardless of whether or not $(g_{u_1}(t_0), g_{u_2}(t_0))$
is on the boundary of the wedge, we see that $(\ol g_{u_1}(t_0), \ol g_{u_2}(t_0))$
is in the interior of $D$ since $g_{u_2}(t_0)=\ol g_{u_2}(t_0)$. 
Hence there exists an $\eps>0$ such that 
$\ol g_{u_1}(t)-\ol f_{u_1}(t)$ does not increase for $t\in[t_0,t_0+\eps]$. 
The quantity $ g_{u_1}(t)- f_{u_1}(t)$ is always non-decreasing and might possibly
increase for some  times in $[t_0,t_0+\eps]$, but in any case
$v(t)$ is non-increasing for $t\in [t_0,t_0+\eps]$. 
The case where $\ol g_{u_1}(t_0)<  g_{u_1} (t_0)$ is treated exactly the same. 
Therefore $v$ is non-increasing on $\{t:v(t)\ne 0\}$. This implies that
$v$ is identically 0.
\end{proof}

\begin{remark}\label{R1-uniq}
Note that we do not assert that $m=\ol m$. This is because it need not be true
in the case $a_1=a_2=1$. Let $g(t)$ be identically $0$ and $f(t)=(-t,-t)$. For $m=(m_1,m_2)$ we can take any choice of non-decreasing continuous functions $m_1(t)$ and
$m_2(t)$ that start at 0 such that $m_1(t)+m_2(t)=t$.

For any matrix $R$ that does not have determinant 0 we can invert $R$, and then
$m$ is uniquely determined once $g$ is from the formula $m=R^{-1}(g-f)$.
\end{remark}

We now turn to non-uniqueness.


\begin{proof}[Proof of Theorem \ref{T2-non}] 
By virtue of Lemma \ref{L1-uniq} we may suppose $a_1=\gamma>1$ and $a_2=1$.
Let $t_n=2^{-n}$, $n=0,1,2, \ldots$.
We first define $f_n, g_n, \ol g_n, m_n$, and $ \ol m_n$ on $[0,2^{-n}]$; these
will be the components
that we use to construct our $f$.

Let $f_n(0)=m_n(0)=\ol m_n(0)=0$ and let
\begin{equation}\label{non-E1}
g_n(0)=(0,\gamma^{-n}), \quad \ol g_n(0)=(\gamma^{-n},\gamma^{-n}).
\end{equation}

Let  $s_{n1}=\tfrac14 2^{-n}$, $s_{n2}=\tfrac12 2^{-n}$,
and $s_{n3}=\tfrac34 2^{-n}$. The functions $f_n, g_n, \ol g_n, m_n$, and
$\ol m_n$ will consist of four pieces.

(1) First $f_n$ moves left a distance $\gamma^{-n}$ at constant speed. 
To be precise, let $f_n(s_{n1})=(-\gamma^{-n}, 0)$.
Let 
$$g_n(s_{n1})=(0,2\gamma^{-n}), \quad \ol g_n(s_{n1})=(0,\gamma^{-n}), \quad 
m_n(s_{n1})=(\gamma^{-n},0), \quad \ol m_n(s_{n1})=0.
$$
Extend the definition of  each of these to $[0,s_{n1}]$ by linear interpolation.

(2) Next $f_n$ moves $\gamma^{-n+1}$ to the right. That is, set
\begin{align*}
f_n(s_{n2})&=(\gamma^{-n+1}-\gamma^{-n},0), \quad g_n(s_{n2})=(\gamma^{-n+1},2\gamma^{-n}), \quad \ol g_n(s_{n2})
=(\gamma^{-n+1},\gamma^{-n}),\\
m_n(s_{n2})&=(\gamma^{-n},0), \quad \ol m_n(s_{n2})=0.
\end{align*} 
Define each of these on $[s_{n1},s_{n2}]$ by linear interpolation.                   

(3) For the third piece, let  $f_n$ move down a distance $2\gamma^{-n}$, which 
 leads to
\begin{align*}
f_n(s_{n3})&=(\gamma^{-n+1}-\gamma^{-n},-2\gamma^{-n}), \quad g_n(s_{n3})=(\gamma^{-n+1},0), \quad \ol g_n(s_{n3})
=(2\gamma^{-n+1},0),\\
m_n(s_{n3})&=(\gamma^{-n},0), \quad \ol m_n(s_{n3})=(0,\gamma^{-n}).
\end{align*}                   
Again, use linear interpolation between $s_{n2}$ and $s_{n3}$.

(4) Finally $f_n$ moves diagonally up and to the left as follows.
Set
\begin{align}
f_n(2^{-n})&=(-\gamma^{-n},\gamma^{-n+1}-2\gamma^{-n}), \quad g_n(2^{-n})=(0,\gamma^{-n+1}), \label{non-E2}\\
\quad \ol g_n(2^{-n})
&=(\gamma^{-n+1},\gamma^{-n+1}), \quad
m_n(s_{n3})=(\gamma^{-n},0), \quad \ol m_n(s_{n3})=(0,\gamma^{-n}).\notag
\end{align}                  
Use linear interpolation once more for values between $s_{n3}$ and $2^{-n}$.

Observe that 
\begin{equation}\label{non-E3}
g_n(t)-g_n(0)=f_n(t)+Rm_n(t)
\end{equation}
 when $t=0, s_{n1}, s_{n2}, 
s_{n3}, 2^{-n}$, and similarly when $g_n$ and $m_n$ are
replaced by $\ol g_n$ and $\ol m_n$.
Since all these
quantities are defined by linear interpolation for times in between the above
values of $t$, we have \eqref{non-E3} holding for all $t\in [0,2^{-n}]$
and the same when $g_n$ and $m_n$ are replaced by $\ol g_n$ and $\ol m_n$.
 
For $t\in (2^{-n}, 2^{-n+1}]$ define
$$f(t)=f_n(t-2^{-n})+\sum_{k=n+1}^\infty f_k(2^{-k}).$$
The series converges by \eqref{non-E2}. 
We use the same formula with $f,f_k,f_n$ replaced by
$m,m_k,m_n$ and $\ol m,\ol m_k, \ol m_n$ to define $m$ and $\ol m$.
For $t\in (2^{-n}, 2^{-n+1}]$ define $g$ by
$$g(t)=g_n(t-2^{-n})$$
and define $\ol g$ similarly.

Given \eqref{non-E3} it is straightforward to check that $g(t)=f(t)+Rm(t)$ and that the analogous
equation holds for $\ol g$.
Clearly $g$ and $\ol g$ are distinct.

The only thing left to check is that $m_j$ increases only when $g_j=0$, $j=1,2$, and the analogous result for $\ol m, \ol g$.
We do this for the first coordinate of $m$, the other cases being almost identical. Note that
the first coordinate of $m$ increases 
only when the first coordinate of one of the $m_n$
increases. By our construction, at these times the first coordinate of $g_n$
is 0, which implies that the first coordinate of $g$ is 0 at these times.
\end{proof}


\end{document}